\documentclass[11pt]{amsart}
\usepackage{amsfonts,amsmath,amssymb,epsfig,cite,graphicx,hyperref,
color, esint, fancyhdr, enumerate, latexsym, amsrefs, mathrsfs}

\newtheorem{thm}{Theorem}[section]
\newtheorem{prop}[thm]{Proposition}
\newtheorem{lem}[thm]{Lemma}
\newtheorem{cor}[thm]{Corollary}

\theoremstyle{definition}
\newtheorem{defn}[thm]{Definition}

\newcommand{\mbb}{\mathbb}
\newcommand{\de}{\delta}
\newcommand{\ov}{\overline}

\newcommand{\pa}{\partial}
\newcommand{\mf}{\mathbb}

\newcommand{\Om}{\Omega}
\newcommand{\al}{\alpha}
\newcommand{\be}{\beta}

\newcommand{\z}{\zeta}

\newcommand{\ti}{\tilde}
\newcommand{\De}{\Delta}
\newcommand{\Ga}{\Gamma}

\renewcommand{\Re}{\operatorname{Re}}
\renewcommand{\Im}{\operatorname{Im}}

\newcommand{\Ric}{\operatorname{Ric}}
\newcommand{\diag}{\operatorname{diag}}

\newcommand{\ad}{\operatorname{ad}}

\numberwithin{equation}{section}
\hypersetup{
    colorlinks,
    citecolor=blue,
    filecolor=black, 
    linkcolor=red,
    urlcolor=black
}

\title{Weighted boundary limits of the Kobayashi--Fuks metric on h-extendible domains}
\keywords{Bergman kernel, Kobayashi--Fuks metric, h-extendible domains}
\subjclass{32F45, 32A25, 32A36}

\author{Debaprasanna Kar}
\address{Department of Mathematics, Indian Institute of Science, Bangalore-560012, India}
\email{debaprasanna@iisc.ac.in}

\begin{document}
\maketitle

\begin{abstract}
We study the boundary behavior of the Kobayashi--Fuks metric on the class of h-extendible domains. Here, we derive the non-tangential boundary asymptotics of the Kobayashi--Fuks metric and its Riemannian volume element by the help of some maximal domain functions and then using their stability results on h-extendible local models.
\end{abstract}

\section{Introduction}

The estimates and the study of asymptotic boundary behavior of invariant metrics play an important role in a wide variety of problems in complex analysis like obtaining regularity results of proper holomorphic mappings, determining biholomorphic equivalence or non-equivalence of domains, describing domains with their non-compact groups of automorphisms, acquiring asymptotic estimates of various classes of holomorphic functions, etc. The limiting behavior of all the classical pseudometrics, e.g. the Carath\'eodory metric, the Bergman metric, the Kobayashi metric, have been studied extensively by many authors both on strongly pseudoconvex domains (see e.g. \cite{Gr, Jarn-Nik-2002, Diederich-1970, Fefferman-1974, Klembeck-1978, Died-Forn, Kim-Park}), as well as on various classes of weakly pseudoconvex domains in $\mbb C^n$ (see e.g. \cite{Nikolov-2000, Lee-2001, Warsz, Yu-1995, BSY-1995, Catlin-1989, Herbort-German-1983, McNeal-1992}). The class of \textit{h-extendible domains} (which will be defined in details later) contains a large collection of weakly pseudoconvex finite type domains. In this note we study the boundary asymptotics of the Kobayashi--Fuks metric, a K\"ahler metric closely related to the Bergman metric, on h-extendible domains. In particular, we compute the non-tangential boundary limits of the Kobayashi--Fuks metric in terms of the corresponding Kobayashi--Fuks invariants of an unbounded local model of the h-extendible domain.

Recall that a $C^{\infty}$-smooth, finite type boundary point $p$ of a domain $\Om$ in $\mbb C^n$ is said to be an h-extendible (or, semiregular) point if $\Om$ is pseudoconvex near $p$ and Catlin and D'Angelo's multitypes of $p$ coincide. If all  boundary points of $\Om$ are h-extendible points, $\Om$ is called an h-extendible domain (or, semiregular domain). As a matter of fact, h-extendible domains are precisely those domains enjoying certain ``bumping property" (which we will describe in the next section). It is known that if $\Om$ is linearly convexifiable near $p$, then $p$ is h-extendible \cite{Conrad}. In particular, convex finite type domains are h-extendible domains. Similarly, $p$ is h-extendible if the Levi form of $p$ has corank at most one \cite{Yu-1995}. Therefore, h-extendibility occurs in the case of strongly pseudoconvex domains in $\mbb C^n$, and on pseudoconvex finite type domains in $\mbb C^2$ as well. In this note all our results are concerned with h-extendible domains which are bounded.

\section{Notations, Definitions and Main Results}

Let us first recall the definition of the Kobayashi--Fuks metric. For a domain $ \Om\subset\mf{C}^n$, the space
\[
A^2( \Om)=\left\{\text{$f:  \Om \to \mf{C}$ holomorphic and $\|f\|^2_ \Om:=\int_{ \Om} \vert f\vert^2 \, dV< \infty$} \right\},
\]
where $dV$ is the Lebesgue measure on $\mf{C}^n$, is a closed subspace of $L^2( \Om)$, and hence is a Hilbert space. It is called the \textit{Bergman space} of $ \Om$. $A^2( \Om)$ carries a reproducing kernel $K_ \Om(z,w)$ called the \textit{Bergman kernel} for $ \Om$. Let $K_ \Om(z):=K_ \Om(z,z)$ be its restriction to the diagonal of $ \Om$. It is well-known (see \cite{Jarn-Pflug-2013}) that 
\begin{align*}
K_ \Om(z)=\sup\big\{|f(z)|^2: f\in A^2( \Om), \|f\|_ \Om= 1\big\}.
\end{align*}
If $ \Om$ is bounded, one easily sees that $K_ \Om>0$. It is also known that $\log K_ \Om$ is a strictly plurisubharmonic function
and thus is a potential for a K\"ahler metric which is called the \textit{Bergman metric} for $ \Om$ and is given by
\[
ds^2_{ \Om}(z)=\sum_{\al,\be=1}^n g^{ \Om}_{\al\ov \be} (z) \, dz_{\al}d\ov z_{\be},
\]
where
\[
g^{ \Om}_{\al \ov \be}(z)=\frac{\pa^2 \log K_{ \Om}}{\pa z_{\al} \pa \ov z_{\be}}(z).
\]
For $z\in  \Om$ and $u\in \mbb C^n$, we denote
\[
G_{ \Om}(z):=\begin{pmatrix}g^{ \Om}_{\al \ov \be}(z)\end{pmatrix}_{n \times n} \quad \text{and} \quad B_ \Om(z,u):=\bigg(\sum_{\al,\be=1}^n g^{ \Om}_{\al\ov \be} (z) \, u_{\al}\ov u_{\be}\bigg)^{1/2},
\]
where $B_ \Om(z,u)$ is the Bergman length of the vector $u$ at $z\in  \Om$. We will denote the components of the Ricci tensor of the Bergman metric on $ \Om$ by $\Ric_{\al \ov \be}^{ \Om}$, and its Ricci curvature by $\text{Ric}_{ \Om}$. Recall that 
\begin{align*}
\Ric_{\al \ov \be}^{ \Om}(z)=  - \frac{\pa^2 \log \det G_ \Om}{\pa z_{\al} \pa \ov z_{\be}}(z)\quad \text{and} \quad \text{Ric}_{ \Om}(z,u)=\frac{\sum_{\al, \be=1}^n \Ric_{\al \ov \be}^{ \Om}(z) u_{\al} \ov u_{\be}}{\sum_{\al,\be=1}^n g^{ \Om}_{\al\ov \be}(z)u_\al \ov u_\be}.
\end{align*}
Kobayashi \cite{Kob59} showed that the Ricci curvature of the Bergman metric on a bounded domain $\Om$ in $\mbb C^n$ is strictly bounded above by $n+1$ and hence the matrix
\[
\ti G_{ \Om}(z):=\begin{pmatrix} \ti g^{ \Om}_{\al \ov \be}(z)\end{pmatrix}_{n \times n} \quad \text{where} \quad \ti g^{ \Om}_{\al \ov \be}(z)=(n+1)g^{ \Om}_{\al\ov \be}(z)-\Ric^{ \Om}_{\al\ov \be}(z),
\]
is positive definite (see also Fuks \cite{Fuks66}). Therefore,
\[
d\ti s^2_{ \Om}=\sum_{\al,\be=1}^n \ti g^{ \Om}_{\al \ov \be}(z)\,dz_{\al}d\ov z_{\be}
\]
is a K\"ahler metric with K\"ahler potential $\log (K_ \Om^{n+1} \det G_{ \Om})$, which we call the \textit{Kobayashi--Fuks metric}. For a domain $ \Om\subset \mbb C^n$, bounded or not, with positive Bergman kernel and positive definite Bergman metric, if the Ricci curvature of the Bergman metric on $ \Om$ is known to be strictly bounded above by $n+1$, then one can define the Kobayashi--Fuks metric $d\ti s^2_{ \Om}$ as a positive definite K\"ahler metric on $ \Om$. Moreover, if $F:  \Om_1 \to  \Om_2$ is a biholomorphism, then
\begin{equation*}\label{tr-kf}
\ti G_{ \Om_1}(z)= F'(z)^t\,\ti G_{ \Om_2}\big(F(z)\big) \ov F'(z),
\end{equation*}
where $F'(z)$ is the complex Jacobian matrix of $F$ at $z$. This implies that $d\ti s^2_{ \Om}$ is an invariant metric on $ \Om$.

The treatment and the study of the Kobayashi--Fuks metric has been fairly new. This metric plays an important role in the study of the Bergman representative coordinates, a tool introduced by Bergman in his program of generalizing the Riemann mapping theorem to $\mbb C^n$ for $n>1$. This fact was first observed by \.{Z}ywomir Dinew \cite{Dinew11}. Dinew, in \cite{Dinew13}, subsequently studied the completeness of this metric on a general class of pseudoconvex domains. The boundary behavior of the Kobayashi--Fuks metric and some of its associated invariants have been obtained in \cite{Borah-Kar-2} on strongly pseudoconvex domains by establishing localizations of this metric near holomorphic peak points. Recently, in \cite{Kar}, by deriving certain boundary estimates for the Kobayashi--Fuks metric on a smoothly bounded planar domain $D$, the existence of geodesic spirals for $d\ti s^2_{D}$ is shown whenever $D$ is non-simply connected.

Now let us describe the class of h-extendible domains, on which we will compute the boundary asymptotics and state the main results. Associated to each smooth boundary point $p$ of a pseudoconvex domain $\Om\subset \mbb C^n$, we have a well-known (invariant) type: the \textit{D'Angelo type} $\De(p)=\left(\De_n(p),\ldots, \De_1(p)\right)$, where the $q$-type $\De_q$, roughly speaking, measures the maximal order of contact of $q$-dimensional varieties with the boundary of $\Om$ at $p$ (see \cite{D'Angelo-1982}). Corresponding to each point $p$ on the smooth boundary $\pa\Om$, D. Catlin \cite{Catlin2} introduced another invariant $\mathcal{M}(p)=(m_1,\ldots,m_n)$, generally called as the \textit{Catlin multitype} of $p$. This invariant arises from the study of the boundary regularity properties of solutions of the $\ov{\pa}$-Neumann problem on finite type domains (in the sense of D'Angelo). Here, the entries $m_i$ can be thought of as the optimal weight assigned to the coordinate direction $z_i$. In \cite{Catlin2}, Catlin also showed that the general relation between these two invariants is $\mathcal{M}(p)\leq \De(p)$ in the sense that $m_{n+1-q}\leq \De_q$ for $1\leq q \leq n$. If $\De_1(p)<\infty$, or in other words if $p$ is a point of finite type (in the sense of D'Angelo), then there exist local coordinates $(z_1,z')=(z_1,\ldots,z_n)$ around $p$ in which $p$ is translated to the origin, and a real-valued, plurisubharmonic, weighted homogeneous polynomial $P$ of degree $\mathcal{M}(p)$ containing no pure terms such that $\Om$ can be defined near $p=0$ locally by (see \cite{BSY-1995})
\begin{align*}
\Re z_1+P(z')+o\bigg(\sum\limits_{j=1}^n |z_j|^{m_j}\bigg)<0.
\end{align*}
Here, the polynomial being \textit{weighted homogeneous} of degree $\mathcal{M}(p)$ means 
\begin{align*}
P\big(\pi_t(z')\big)=tP(z'),\quad\text{where}\quad \pi_t(z)=\big(t^{1/m_1}z_1, \ldots, t^{1/m_n}z_n\big).
\end{align*}
In the above expression, $\pi_t$ is an anisotropic dilation acting on $\mbb C^n$ and $\pi_t(z')$ is its restriction to the last $n-1$ coordinates. The unbounded domain
\begin{align*}
D_0=\{(z_1, z'):\Re z_1+P(z')<0\}
\end{align*}
is called the \textit{local model} for $\Om$ at $p$ in the coordinates $(z_1,z')$. Note that the polynomial $P$, and hence the local model $D_0$, is not unique in general and depends on the boundary point $p$.

\begin{defn}
A boundary point $p$ of $\Om$ is called \textit{h-extendible} if $\mathcal{M}(p)=\De(p)<\infty$. A pseudoconvex domain $\Om$ is called an \textit{h-extendible domain} if all its boundary points are h-extendible.
\end{defn}

\noindent It is shown in \cite{Died-Herbort-1994} (see also \cite{Yu-1994}) that $\Om$ is h-extendible at $p$ if and only if the local model $D_0$ admits a \textit{bumping function} $a(z')$ on $\mbb C^{n-1}$ which satisfies the following conditions:
\begin{itemize}
\item[(i)] $a(z')$ is $C^{\infty}$-smooth and positive whenever $z'\neq 0$,
\item[(ii)] $a(z')$ is weighted homogeneous of degree $\mathcal{M}(p)$ (the same weight as for $P$), and
\item[(iii)] $P(z')-\de a(z')$ is strictly plurisubharmonic on $\mbb C^{n-1}\setminus \{0\}$ for all $0<\de\leq 1$.
\end{itemize}
In such a case, the local model $D_0$ is called an \textit{h-extendible model}. The above three conditions basically state that $D_0$ can be approximated from outside by the family of pseudoconvex domains $D_{\de}:=\{(z_1,z'): \Re z_1+P(z')-\de a(z')<0\}$ having the same homogeneity as $D_0$. Therefore one can say that the local model is ``homogeneously extendible" to a larger pseudoconvex domain, which inspired the terminology ``h-extendible". We will call each $D_{\de}$ as a \textit{bumped model} for $\Om$ at $p$. From the bumping property above one easily sees that any smoothly bounded strongly pseudoconvex domain is an h-extendible domain, as $P$ is strictly plurisubharmonic and serves as its own bumping function in this case.

Before stating our results, let us fix some notations. We denote by $\ti g_{\Om}(z)$ the determinant of $\ti G_{\Om}(z)$. The length of a vector $u$ at a point $z\in \Om$ in $d\ti s^2_{\Om}$ will be denoted by $\ti B_{\Om}(z,u)$, i.e.,
\[
\ti B_{\Om}(z,u)=\bigg(\sum_{\al,\be=1}^n \ti g^{\Om}_{\al\ov \be} (z) \, u_{\al}\ov u_{\be}\bigg)^{1/2}.
\]
The distance of $z \in \Om$ to the boundary $\pa \Om$ will be denoted by $d_{\Om}(z)$. We will drop the subscript $\Om$ and simply write as $d(z)$ as the domain under consideration is clear.

\begin{thm}\label{kf metric}
Let $\Om$ be a bounded pseudoconvex domain in $\mbb C^n$ and $p\in \pa \Om$ an h-extendible boundary point with Catlin's multitype $(m_1,\ldots,m_n)$. Then there are local coordinates $z$, and a local model $D_0$ at $p$ such that
\begin{align*}
\lim_{\substack{z\to p \\ z\in \Ga}}\big|\pi_{1/d(z)}(u)\big|^{-1}\ti B_{\Om}(z,u)=\ti B_{D_0}(b^*,u^*).
\end{align*}
Here $\Ga$ is a non-tangential cone in $\Om$ with vertex at $p$, $\pi_{1/d(z)}(u)=\big(d(z)^{-1/m_1}u_1,\ldots,d(z)^{-1/m_n}u_n\big) $, the unit vector $u^*=\lim_{z\to p} \pi_{1/d(z)}(u)/|\pi_{1/d(z)}(u)|$ and $b^*=(-1,0')$.
\end{thm}

\begin{thm}\label{g tilde}
Under the hypotheses and notations of Theorem~\ref{kf metric},
\begin{align*}
\lim_{\substack{z\to p \\ z\in \Ga}} \big(d(z)\big)^{\sum_{j=1}^n 2/m_j} \ti g_{\Om}(z)=\ti g_{D_0}(b^*).
\end{align*}
\end{thm}

Note that in the theorems above, although $\Om$ is required to be pseudoconvex globally, we only need the boundary $\pa \Om$ to be smooth and finite type in a neighborhood of $p$. Again, observe that earlier we defined the Kobayashi--Fuks metric as a positive definite K\"ahler metric on bounded domains, hence a priori it is not obvious why $\ti B_{D_0}$ and $\ti g_{D_0}$ are non-zero. In the next section, by using certain maximal domain function and a result of Boas, Straube and Yu, we see that whenever $z\in D_0$ and $u\neq 0$, $\ti B(z,u)$ and $\ti g(z)$ are positive not only on $D_0$, but also on the bumped model $D_{\de}$ for all sufficiently small $\de>0$.

In the above theorems, we evaluate the boundary limits in terms of their values at a fixed interior point of the local model. Explicit computation of the Kobayashi--Fuks metric, in general, is quite arduous because of the complexity of the computation involved. Even for domains with known Bergman kernel and metric (e.g. the annulus, the symmetrized bidisc), it is very hard to compute the Kobayashi--Fuks metric. However, we can compute the boundary limits in these theorems explicitly if the local model is simple. One such case is when $\Om$ is a smoothly bounded strongly pseudoconvex domain:

\begin{cor}\label{str-pscvx}
Let $\Om\subset \mbb C^n$ be a smoothly bounded strongly pseudoconvex domain and $p\in \pa \Om$. Then there are holomorphic coordinates $z$ near $p$ in which
\begin{align*}
\lim\limits_{z\to p} \big(d(z)\big)^{n+1}\ti g_{\Om}(z) &= \dfrac{(n+1)^n (n+2)^n}{2^{n+1}},\\
\lim\limits_{z\to p} \dfrac{d(z)}{\sqrt{|u_1|^2+d(z)|u'|^2}} \ti B_{\Om}(z,u) &=\dfrac{1}{2}\sqrt{(n+1)(n+2)}\sqrt{|u^*_1|^2+2|(u^*)'|^2}
\end{align*}
for all $u\in \mbb C^n\setminus\{0\}$. Here $u^*=\lim_{z\to p}\ti u(z)/|\ti u(z)|$ with $\ti u(z)=\big(u_1/d(z),u'/\sqrt{d(z)}\big)$.
\end{cor}

Note that the limits in Corollary~\ref{str-pscvx} have been obtained previously in \cite{Borah-Kar-2} by implementing Pinchuk's scaling method near strongly pseudoconvex boundary points. Another instance in which we can compute the limits in Theorems~\ref{kf metric} and \ref{g tilde} explicitly is when the local model has certain circular symmetry:

\begin{cor}\label{circular}
Let $\Om$ be a smoothly bounded pseudoconvex domain in $\mbb C^2$ and $p\in \pa \Om$. Suppose that the local model for $\Om$ at $p$ is given by 
\[
D_0:=\{(z_1,z_2):\Re z_1+P(z_2)<0\}
\]
in some local coordinates $z$, where $P$ is a circular polynomial in the sense that $P(e^{i\theta}z_2)=P(z_2)$ for all $z_2 \in \mbb C$ and $\theta \in \mbb R $. Then $D_0$ is biholomorphic to the bounded Reinhardt domain $E:=\{(z_1,z_2): |z_1|^2+P\big(|z_2|\big)-1<0\}$ and, with the notations of Theorems~\ref{kf metric} and \ref{g tilde},
\begin{align*}
&\lim_{\substack{z\to p \\ z\in \Ga}}\big|\pi_{1/d(z)}(u)\big|^{-1}\ti B_{\Om}(z,u)=\dfrac{1}{2\sqrt{b_1 b_2}} \bigg[\big(b_1 c_{11}+b_2 c_{20}\big)\big|u^*_1\big|^2+ 4\big(b_2 c_{11}+b_1 c_{02}\big)\big|u^*_2\big|^2\bigg]^{1/2},\\
&\lim_{\substack{z\to p \\ z\in \Ga}} \big(d(z)\big)^{\sum_{j=1}^n 2/m_j} \ti g_{\Om}(z)=\dfrac{1}{4 b_1^2 b_2^2}\big(b_1 c_{11}+b_2 c_{20}\big)\big(b_2 c_{11}+b_1 c_{02}\big).
\end{align*}
Here $b_i=1/\|z_i\|_E^2$ and $c_{jk}=1/\|z_1^j z_2^k\|_E^2$ for non-negative integers $i, j, k$ with $1\leq i \leq 2$ and $j+k=2$.
\end{cor}
Note that the computations of the Bergman metric, along with its holomorphic sectional curvature, on local model domains with circular symmetries in higher dimensions have been carried out in \cite[p.~458]{BSY-1995}.

\section{Maximal domain functions and stability}\label{MDF}

We begin by recalling the maximal domain functions introduced by Krantz and Yu: For a domain $\Om \subset \mathbb{C}^n$, $z \in \Om$ and a non-zero vector $u \in \mathbb{C}^n$, let
\begin{align*} \label{I}
I_{\Om}(z,u)&=\sup \left\{u^tf''(z){\ov G}^{-1}_{\Om}(z)\ov{f''}(z)\ov u: \|f\|_{\Om}=1, f(z)=f'(z)=0\right\},\\
M_{\Om}(z,u)&=\sup \big\{K_{\Om}^{n-1}(z)\,u^tf''(z)\ov \ad G_{\Om}(z)\ov{f''}(z)\ov u: \|f\|_{\Om}=1, f(z)=f'(z)=0\big\}. \nonumber
\end{align*}
Here $\ad A$ is the adjoint of a matrix $A$ and $f''(z)$ is the symmetric matrix
\[
f''(z)=\begin{pmatrix}\frac{\pa^2 f}{\pa z_i \pa z_j}(z)\end{pmatrix}_{n \times n}.
\]
Clearly
\begin{align*}
M_{\Om}(z,u)=K_{\Om}^n(z)J_{\Om}(z)I_{\Om}(z,u),
\end{align*}
where $J_{\Om}$ is the \textit{Bergman canonical invariant} defined by
\begin{align*}
J_{\Om}(z)=\dfrac{\det G_{\Om}(z)}{K_{\Om}(z)}.
\end{align*}
Although the domain function $I_\Om$ may not be monotonic, however $M_{\Om}$ can be shown to be a monotonically decreasing function with increasing domain (see Proposition~2.2 of \cite{Krantz-Yu}). It is shown in Proposition~2.1 of \cite{Krantz-Yu} that whenever $K_{\Om}(z)$ and $B_{\Om}(z,u)$ are positive,
\begin{align*}
\Ric_{\Om}(z,u)=(n+1)-\frac{1}{B^2_{\Om}(z,u)K_{\Om}(z)}I_{\Om}(z,u)\\
=(n+1)-\frac{1}{B^2_{\Om}(z,u)K_{\Om}^{n+1}(z)J_{\Om}(z)}M_{\Om}(z,u). \nonumber
\end{align*}
Therefore, one can show using the definition of the Kobayashi--Fuks metric that
\begin{align}\label{eq 3}
\ti B_{\Om}^2(z,u)=\dfrac{I_\Om(z,u)}{K_\Om(z)}=\dfrac{M_\Om(z,u)}{K_\Om^{n+1}(z)J_{\Om}(z)}.
\end{align}
Under a biholomorphic map $F:\Om_1\to \Om_2$, the domain function $I_{\Om}$ transforms similar to the Bergman kernel $K_{\Om}$, i.e., $I_{\Om_1}(z,u)=I_{\Om_2}\big(F(z),F'(z)u\big)\big|\det F'(z)\big|^2$. Therefore,
\begin{align}\label{trans M}
M_{\Om_1}(z,u)=M_{\Om_2}\big(F(z),F'(z)u\big)\big|\det F'(z)\big|^{2(n+1)}.
\end{align}

Now let us affirm the claim we made after Theorem~\ref{g tilde} that $\ti B_{D_0}(z,u), \ti g_{D_0}(z)>0$ for all $z\in D_0$ and non-zero vector $u\in \mbb C^n$. We also show that $\ti B_{D_\de}(z,u)$ and $\ti g_{D_\de}(z)$ are strictly positive for $z\in D_0, u\neq 0$ and $0<\de\leq 1/2$. Basically all these properties hold because h-extendible models support a large number of square integrable holomorphic functions. In fact, Boas, Straube and Yu in \cite[p.~454]{BSY-1995} proved the following:

\begin{lem}\label{taylor poly}
For a fixed point $\zeta$ in $D_0$, a positive integer $m$, and a holomorphic polynomial $q$, there exists a function $f$ in $A^2(D_{1/2})$ such that $f(z)=q(z)+O(|z-\zeta|^m)$ as $z\to \zeta$.
\end{lem}
\noindent Using this lemma to the following supremums
\begin{align*}
I_{\Om}'(\zeta)&= \sup\big\{|f(\zeta)|^2: f\in A^2(\Om), \|f\|_ \Om =1\big\},\\
I''_{\Om}(\zeta, u)&= \sup\bigg\{\bigg|\sum_{j=1}^n u_j \dfrac{\pa f}{\pa z_j}(\zeta)\bigg|^2: f\in A^2(\Om), \|f\|_ \Om =1, f(\zeta)=0\bigg\},
\end{align*}
one easily sees that $I'_{D_0}, I''_{D_0}, I'_{D_\de}, I''_{D_\de}$ are strictly positive on $D_0$ for $0<\de\leq 1/2$. Consequently, using the following relations (see \cite{Jarn-Pflug-2013})
\begin{align*}
K_{\Om}(\zeta)= I_{\Om}'(\zeta), \quad B_{\Om}(\zeta, u)=\bigg(\dfrac{I''_{\Om}(\zeta, u)}{I_{\Om}'(\zeta)}\bigg)^{1/2},
\end{align*}
we conclude that $K_{D_0}, K_{D_\de}, B_{D_0}$ and $B_{D_\de}$ are non-zero on $D_0$. Again, using similar arguments, by implementing Lemma~\ref{taylor poly} in the definition of the maximal domain function $I_\Om$ and by virtue of the relation $\ti B^2_{\Om}=I_\Om/K_\Om$, one obtains strict positivity of $\ti B_{D_0}(z,u), \ti B_{D_\de}(z,u), \ti g_{D_0}(z), \ti g_{D_\de}(z)$ for $z \in D_0$ and non-zero vector $u$.

Next, we need some stability results for these maximal domain functions, along with some other quantities, under perturbations of the domain. Although the full scope of Ramadanov type result, which is commonly used in Pinchuk's scaling of strongly pseudoconvex domains (see e.g. \cite{BV-ns}), is not available here, an ad hoc stability result on unbounded model domains will suffice in our case. By slightly modifying Ramadanov's arguments in \cite{Ramadanov-1967}, one can show that if a sequence of domains $\Om_j$ converges to a limit domain $\Om$ normally from inside in the sense that $\Om_j\subset \Om$ and any compact set of $\Om$ is contained in $\Om_j$ for all large $j$, then the Bergman kernel of $\Om_j$ converges to the Bergman kernel of $\Om$ along with all its partial derivatives uniformly on compact subsets of $\Om$. Since the functions $J_{\Om_j}, M_{\Om_j}, \ti g_{\Om_j}$ are expressible in terms of the Bergman kernel and its derivatives (see Proposition~2.1 of \cite{Krantz-Yu}), by the previous remark they will converge respectively to $J_\Om, M_\Om, \ti g_\Om$ uniformly on compacts.

For our purpose we will need another stability result for the quantities $J, M$ and $\ti g$ under the local Hausdorff convergence of the bumped models $D_\de$ to $D_0$ as $\de\to 0^+$. In \cite{BSY-1995}, the authors proved that $K_{D_{\de}}(z)\to K_{D_0}(z)$ uniformly on compact subsets of $D_0$ as $\de\to 0^+$. Using a normal family argument one can then easily deduce the stability of all higher order derivatives of the Bergman kernel. As we remarked earlier, since all the quantities we are studying here can be expressed in terms of derivatives of the Bergman kernel, one obtains the following results (see also Proposition~2.5 of \cite{Krantz-Yu}):

\begin{prop}\label{stability}
Let $D_0=\{(z_1,z')\in \mbb C^n: \Re z_1+ P(z')<0\}$ be an h-extendible model in $\mbb C^n$ and $D_{\de}=\{(z_1,z')\in \mbb C^n: \Re z_1+ P(z')-\de a(z')<0\}$, where $a(z')$ is a bumping function for $P(z')$. Then, for $z\in D_0$ and $u\in \mbb C^n$,
\begin{align*}
\lim_{\de\to 0^+} J_{D_{\de}}(z) &=J_{D_0}(z),\\
\lim_{\de\to 0^+} M_{D_{\de}}(z,u) &=M_{D_0}(z,u),\\
\lim_{\de\to 0^+} \ti g_{D_{\de}}(z) &=\ti g_{D_0}(z).
\end{align*}
Moreover, the second convergence above is uniform on compact subsets of $D_0\times \mbb C^n$, whereas the remaining convergences are uniform on compact sets of $D_0$.
\end{prop}

Before proving our theorems, let us see some localization results that would reduce the computation of boundary asymptotics to a small neighborhood of our h-extendible boundary point.

\begin{lem}[Localization]\label{localization}
Let $\Om \subset \mf{C}^n$ be a bounded pseudoconvex domain with a holomorphic peak point $p\in \pa \Om$. If $U$ is a sufficiently small neighborhood of $p$, then
\begin{align*}
\lim_{z\to p} \dfrac{J_\Om (z)}{J_{\Om \cap U}(z)}=\lim_{z\to p} \dfrac{M_\Om (z,u)}{M_{\Om \cap U}(z,u)}=\lim_{z\to p} \dfrac{\ti g_\Om (z)}{\ti g_{\Om \cap U}(z)}=1.
\end{align*}
\end{lem}
Proofs of the first two localizations can be found in \cite{Krantz-Yu}, whereas the last one is proved in \cite{Borah-Kar-2}. Since there exist local peak functions at h-extendible boundary points (see \cite{Died-Herbort-1994, Yu-1994}), the above localization results can be applied to any bounded h-extendible domain $\Om$.

\section{Proofs of the results}

The idea for proving Theorem~\ref{kf metric} and Theorem~\ref{g tilde} are through similar lines. After bringing our domain $\Om$ into a normal form, we first blow up a small neighborhood of the point $p\in \pa \Om$ via a family of dilation maps which would transfer our boundary problem into interior problem on the dilated domains. Then applying a crucial fact that the dilated domains are contained in bumped model domains $D_{\de}$ and using monotonicity, we deduce the boundary asymptotics on the previously mentioned small neighborhood of $p$. Finally, utilizing the localization results, we obtain the required asymptotics on $\Om$.

Let us first make an initial change of coordinates near $p$ to bring the defining function of $\Om$ into a normal form. This coordinate change does not alter our non-tangential approach to $p$. By \cite{Yu-1995}, we can choose local holomorphic coordinates near $p$ in which $p=0$, and the domain $\Om$ is defined near $p$ by the equation $r<0$ with
\begin{align*}
r(z_1,z')=\Re z_1+ P(z')+O\big((\Im z_1)^2+\sigma(z')^{1+\al}\big),
\end{align*}
where $P$ is a weighted homogeneous polynomial as described earlier, $\al$ is a positive constant and $\sigma(z')=\sum_{j=2}^n |z_j|^{m_j}$. Let $a$ be a bumping function for $P$. After a further change of variables (see \cite[p.~456]{BSY-1995}), which again doesn't change our non-tangential approach to the origin, we may assume we have the following situation: $\Om$ has a local model $D_0=\{z:\Re z_1+P(z')<0\}$ for $\Om$ at $0$, and for each $\de\in (0,1)$, there is a neighborhood $U_{\de}$ of the origin such that the bumped model $D_\de=\{z:\Re z_1+P(z')-\de a(z')<0\}$ contains $\Om\cap U_{\de}$. For the time being let us fix a $\de$ in $(0,1)$ and the neighborhood $U_\de$.

\begin{proof}[Proof of Theorem~\ref{kf metric}]
For $z\in \Om\cap U_\de$, consider the anisotropic dilation map $\pi_{1/|r(z)|}$ and apply it on $\Om \cap U_{\de}$. We will use original notations of Boas, Straube and Yu as in \cite{BSY-1995} to denote the dilated domain by $\Om_z^{\de}:= \pi_{1/|r(z)|}(\Om\cap U_{\de})$ and the image of $z$ under this dilation by $\zeta(z):=\pi_{1/|r(z)|}(z)$. The notation $\Om^{\de}_z$ signifies the dependency of the dilated domains on $\de\in (0,1)$ and $z\in \Om\cap U_\de$. By the transformation rule of $M_\Om$ as in (\ref{trans M}), we have 
\begin{align*}
M_{\Om\cap U_\de}(z,u)=M_{\Om_z^\de}\big(\z(z),\z_z(u)\big)\big|\det \pi'_{1/|r(z)|}(z)\big|^{2(n+1)},
\end{align*}
which is same as
\begin{align}\label{eq 1}
\big|r(z)\big|^{(n+1)\sum_{j=1}^n 2/m_j}M_{\Om\cap U_\de}(z,u)=M_{\Om_z^\de}\big(\z(z),\z_z(u)\big).
\end{align}
Here $\z_z(u)$ denotes the vector $\pi_{1/|r(z)|}(u)$. Using the monotonicity property of the domain function $M_\Om$, we obtain
\begin{align}\label{ineq 1}
\limsup_{\substack{z\to 0 \\ z\in \Ga}} M_{\Om_z^\de}\big(\z(z),\z_z(u)\big)\leq \limsup_{\substack{z\to 0 \\ z\in \Ga}} M_{\Om_z^\de \cap D_0} \big(\z(z),\z_z(u)\big).
\end{align}
As $z\to 0$, the sequence of domains $\Om_z^\de \cap D_0$ converges to the local model $D_0$ from inside, and hence, by the comments made in Section~\ref{MDF}, $M_{\Om_z^\de \cap D_0}(\z,v)$ converges to $M_{D_0}(\z,v)$ uniformly on compact subsets of $D_0\times \mbb C^n$. When $z\to 0$ in a cone, the point $\z(z)$ approaches a compact portion of the real line $\{\Re z_1=-1, z'=0\}$. Note that $M_{D_0}$ is independent of the $\Im z_1$ variable, since $D_0$ is, and hence
\begin{align}\label{eq 2}
\lim_{\substack{z\to 0 \\ z\in \Ga}}\big|\z_z(u)\big|^{-2} M_{\Om_z^\de \cap D_0} \big(\z(z),\z_z(u)\big)=M_{D_0}(b^*,u^*),
\end{align}
where $u^*=\lim_{z\to 0} \z_z(u)/|\z_z(u)|$. Note that the vector $u^*$ agrees with its description as given in Theorem~\ref{kf metric} since $d(z)/|r(z)|\to 1$ as $z$ approaches 0 in the non-tangential region $\Gamma$. Now, combining equations (\ref{eq 1}), (\ref{ineq 1}), (\ref{eq 2}) and using the localization lemma for $M_\Om$, we obtain
\begin{align}\label{ineq 2}
\limsup_{\substack{z\to 0 \\ z\in \Ga}} \big|r(z)\big|^{(n+1)\sum_{j=1}^n 2/m_j}\big|\z_z(u)\big|^{-2} M_{\Om}(z,u)\leq M_{D_0}(b^*,u^*).
\end{align}

On the other hand, as $\Om_z^\de\subset D_\de$, we have $M_{\Om_z^\de}\big(\z(z),\z_z(u)\big)\geq M_{D_\de}\big(\z(z),\z_z(u)\big)$. Since $D_{\de}$ is also independent of the $\Im z_1$ variable, and because $M_{D_\de}(\cdot,\cdot)$ is a continuous function on $D_\de\times \mbb C^n$, one concludes
\begin{align*}
\liminf_{\substack{z\to 0 \\ z\in \Ga}}\big|\z_z(u)\big|^{-2} M_{\Om_z^\de}\big(\z(z),\z_z(u)\big)\geq M_{D_\de}(b^*,u^*).
\end{align*}
Using (\ref{eq 1}) and the localization lemma for $M_\Om$ in the above equation, we obtain
\begin{align*}
\liminf_{\substack{z\to 0 \\ z\in \Ga}} \big|r(z)\big|^{(n+1)\sum_{j=1}^n 2/m_j}\big|\z_z(u)\big|^{-2} M_{\Om}(z,u)\geq M_{D_\de}(b^*,u^*).
\end{align*}
The stability result for $M_{D_\de}$, as in Proposition~\ref{stability}, then implies
\begin{align}\label{ineq 3}
\liminf_{\substack{z\to 0 \\ z\in \Ga}} \big|r(z)\big|^{(n+1)\sum_{j=1}^n 2/m_j}\big|\z_z(u)\big|^{-2} M_{\Om}(z,u)\geq M_{D_0}(b^*,u^*).
\end{align}
Eqs.~(\ref{ineq 2}) and (\ref{ineq 3}), along with the fact that the ratio of $d(z)$ and $|r(z)|$ has limit 1 as $z\to 0$ in a non-tangential cone, together imply
\begin{align}\label{eq 5}
\lim_{\substack{z\to 0 \\ z\in \Ga}} \big(d(z)\big)^{(n+1)\sum_{j=1}^n 2/m_j}\big|\pi_{1/d(z)}(u)\big|^{-2} M_{\Om}(z,u)= M_{D_0}(b^*,u^*).
\end{align}
Moreover, it is shown in \cite{Krantz-Yu} that
\begin{align}\label{eq 6}
\lim_{\substack{z\to 0 \\ z\in \Ga}}\big(d(z)\big)^{(n+1)\sum_{j=1}^n 2/m_j} K_\Om^{n+1}(z)=K_{D_0}^{n+1}(b^*) \quad \text{and}\quad \lim_{\substack{z\to 0 \\ z\in \Ga}} J_\Om(z)=J_{D_0}(b^*).
\end{align}
Therefore, using (\ref{eq 5}) and (\ref{eq 6}) in the expression (\ref{eq 3}), we obtain the desired limit.
\end{proof}

We now establish a monotonicity result involving $\ti g_\Om$ which will help us proving Theorem~\ref{g tilde}. 
\begin{lem}
The function $(K_\Om)^{2n}J_\Om\ti g_\Om$ monotonically decreases with increasing domain. That is, for two domains $\Om_1\subset \Om_2$ in $\mbb C^n$ with $z\in \Om_1$, we have
\begin{align*}
\big(K_{\Om_2}(z)\big)^{2n}J_{\Om_2}(z)\ti g_{\Om_2}(z)\leq \big(K_{\Om_1}(z)\big)^{2n}J_{\Om_1}(z)\ti g_{\Om_1}(z).
\end{align*}
\end{lem}

\begin{proof}
By Lemma~3.2 of \cite{Borah-Kar-2}, there exist non-singular matrix $Q=Q(z)$ and positive real numbers $d_1(z),\ldots,d_n(z)$ such that
\begin{align*}
Q^t(z)\ti G_{\Om_2}(z)\ov Q(z)=\diag\{d_1(z),\ldots,d_n(z)\}\quad \text{and}\quad Q^t(z)\ti G_{\Om_1}(z)\ov Q(z)=\mathbb{I},
\end{align*}
where `$\diag$' stands for the diagonal matrix and $\mbb{I}$ for the identity matrix. Taking determinant above on both sides yields
\begin{align*}
\dfrac{\ti g_{\Om_2}(z)}{\ti g_{\Om_1}(z)}=\prod\limits_{j=1}^n d_j(z).
\end{align*}
Note that (\ref{eq 3}) implies
\begin{align*}
\ti B^2_{\Om_2}(z,u)=\dfrac{K_{\Om_1}(z)}{K_{\Om_2}(z)}\dfrac{I_{\Om_2}(z,u)}{I_{\Om_1}(z,u)}\ti B^2_{\Om_1}(z,u)\quad \text{for every}\,\,u\in \mbb C^n.
\end{align*}
For $j$-th standard unit vector $e_j=(0,\ldots,1,\ldots,0)$, putting $u=Q(z)e_j$ in the previous expression, we get
\begin{align*}
d_j(z)=\dfrac{K_{\Om_1}(z)}{K_{\Om_2}(z)}\dfrac{I_{\Om_2}\big(z,Q(z)e_j\big)}{I_{\Om_1}\big(z,Q(z)e_j\big)}\quad \text{for}\enskip j=1,\ldots,n.
\end{align*}
Therefore,
\begin{align*}
\dfrac{\ti g_{\Om_2}(z)}{\ti g_{\Om_1}(z)}&=\left(\frac{K_{\Om_1}(z)}{K_{\Om_2}(z)}\right)^n\,\prod\limits_{j=1}^n \frac{I_{\Om_2}\big(z,Q(z)e_j\big)}{I_{\Om_1}\big(z,Q(z)e_j\big)}\\
&=\left(\frac{K_{\Om_1}(z)}{K_{\Om_2}(z)}\right)^{2n} \dfrac{J_{\Om_1}(z)}{J_{\Om_2}(z)}\prod\limits_{j=1}^n \frac{M_{\Om_2}\big(z,Q(z)e_j\big)}{M_{\Om_1}\big(z,Q(z)e_j\big)},
\end{align*}
which implies,
\begin{align*}
\dfrac{\big(K_{\Om_2}(z)\big)^{2n}J_{\Om_2}(z)\ti g_{\Om_2}(z)}{\big(K_{\Om_1}(z)\big)^{2n}J_{\Om_1}(z)\ti g_{\Om_1}(z)}=\prod\limits_{j=1}^n \frac{M_{\Om_2}\big(z,Q(z)e_j\big)}{M_{\Om_1}\big(z,Q(z)e_j\big)}\leq 1.
\end{align*}
The last inequality above follows since $M_\Om$ is monotonically decreasing.
\end{proof}

\begin{proof}[Proof of Theorem~\ref{g tilde}]
For a domain $D\subset \mbb C^n$ and $z\in D$, let us denote $$T_{D}(z):=\big(K_{D}(z)\big)^{2n}J_{D}(z)\ti g_{D}(z).$$
Since the transformation rule of $\ti g_{D}$ is similar to that of $K_D$, one can check, under biholomorphism $F:D_1\to D_2$,
\begin{align*}
T_{D_1}(z)=T_{D_2}\big(F(z)\big)\big|\det F'(z)\big|^{2(2n+1)}.
\end{align*}
Therefore,
\begin{align*}
\big|r(z)\big|^{(2n+1)\sum_{j=1}^n 2/m_j}T_{\Om\cap U_\de}(z)=T_{\Om_z^\de}\big(\z(z)\big).
\end{align*}
Since the localization of $K_\Om$ is well-known and localizations of $J_\Om, \ti g_{\Om}$ are given in Lemma~\ref{localization}, one easily sees that $T_\Om$ can be localized near a holomorphic peak point of a bounded pseudoconvex domain. Therefore proceeding in the similar lines as in the proof of Theorem~\ref{kf metric}, replacing the maximal domain function $M$ by the quantity $T$, one arrives at
\begin{align}\label{eq 4}
\lim_{\substack{z\to 0 \\ z\in \Ga}} \big(d(z)\big)^{(2n+1)\sum_{j=1}^n 2/m_j}T_{\Om}(z)= T_{D_0}(b^*).
\end{align}
Again, making use of the boundary limits 
\begin{align*}
&\lim_{\substack{z\to 0 \\ z\in \Ga}}\big(d(z)\big)^{\sum_{j=1}^n 2/m_j} K_\Om(z)=K_{D_0}(b^*),\\
&\lim_{\substack{z\to 0 \\ z\in \Ga}} J_\Om(z)=J_{D_0}(b^*)
\end{align*}
in (\ref{eq 4}), we conclude
\begin{align*}
\lim_{\substack{z\to p \\ z\in \Ga}} \big(d(z)\big)^{\sum_{j=1}^n 2/m_j} \ti g_{\Om}(z)=\ti g_{D_0}(b^*).
\end{align*}
This proves our claim.
\end{proof}

Once Theorems~\ref{kf metric} and \ref{g tilde} have been proved, Corollary~\ref{str-pscvx} follows almost immediately with the following set of observations:
\begin{itemize}
\item[$\bullet$] If $p\in \pa \Om$ is a smooth strongly pseudoconvex boundary point, we may choose holomorphic coordinates near $p$ in which the local model $D_0$ is the Siegel half-space $\{z\in \mbb C^n:2\Re z_1+|z'|^2<0\}$.
\item[$\bullet$] The Cayley transform $\Phi$ given by
\begin{equation*}\label{C-T}
\Phi(z_1,\ldots,z_n)=\left(\dfrac{z_1+1}{z_1-1}, \dfrac{\sqrt 2\,z'}{z_1-1}\right)
\end{equation*}
maps $D_{0}$ biholomorphically onto the unit ball $\mbb{B}^n$. Also, $\Phi(b^{*})=\Phi\big((-1,0')\big)=0$ and $\Phi'(b^*)=-\diag\big\{1/2, 1/\sqrt 2,\ldots,1/\sqrt 2\big\}.$
\item[$\bullet$] The Catlin multitype of a smooth strongly pseudoconvex boundary point $p$ is given by $\mathcal{M}(p)=(1,2,\ldots,2)$.
\item[$\bullet$] Using the transformation rules of $\ti B$ and $\ti g$ under biholomorphism, we have
\begin{align*}
\ti B_{D_0}(b^*, u^*)=\ti B_{\mbb B^n}\big(0,\Phi'(b^*)u^*\big), \quad \ti g_{D_0}(b^*)=\ti g_{\mbb B^n}(0)\big|\det \Phi'(b^*)\big|^2.
\end{align*}
\item[$\bullet$] The Kobayashi--Fuks metric on $\mbb B^n$ at the origin is given by 
\[
d\ti s^2_{\mathbb{B}^n}(0)= (n+1)(n+2)\sum_{\al,\be=1}^n \de_{\al\ov\be} dz_{\al} d\ov z_{\be}.
\] 
\end{itemize} 
One can refer to \cite{Borah-Kar-2} for the computation of the Kobayashi--Fuks metric on $\mbb B^n$. \hfill \qed

\begin{proof}[Proof of Corollary~\ref{circular}]
We first observe that $p$ is indeed an h-extendible boundary point of $\Om$. The polynomial $P$ here is weighted homogeneous, subharmonic in the variable $z_2$ and has circular symmetry. The mean value peoperty then implies that $P$ is positive everywhere except at the origin. Therefore $P$ serves as its own bumping function, which implies that the model $D_0$ is h-extendible. Consequently, Theorems~\ref{kf metric} and \ref{g tilde} can be applied to compute boundary limits on $\Om$ near $p$.

Next, note that the mapping
\begin{align*}
\Psi(z_1,z_2):=\bigg(\dfrac{z_1+1}{z_1-1}, \dfrac{4^{1/m_2}z_2}{(z_1-1)^{2/m_2}} \bigg),
\end{align*}
where $m_2$ is the D'Angelo 1-type of $p\in \pa \Om$, takes $D_0$ biholomorphically onto $E$ with $b^*=(-1,0)$ being mapped to the origin. To see this, we denote the defining functions of $D_0$ and $E$ respectively by
\begin{align*}
\rho_{D_0}(z):=\Re z_1+P(z_2), \quad \rho_E(z):=|z_1|^2+P\big(|z_2|\big)-1.
\end{align*}
With observations that $P(z_2)=P\big(|z_2|\big)$ and 
\[
\dfrac{4^{1/m_2}z_2}{(z_1-1)^{2/m_2}}=\pi_{s}(z_2) \quad \text{with}\quad s=\dfrac{4}{(z_1-1)^2},
\]
we obtain
\begin{align*}
\rho_E &\circ \Psi(z)=\bigg|\dfrac{z_1+1}{z_1-1}\bigg|^2+P\bigg(\bigg|\dfrac{4^{1/m_2}z_2}{(z_1-1)^{2/m_2}}\bigg|\bigg)-1 \\
&=\dfrac{4 \Re z_1}{|z_1-1|^2}+ \dfrac{4}{|z_1-1|^2} P(z_2)=\dfrac{4}{|z_1-1|^2}\, \rho_{D_0}(z),
\end{align*}
which shows the biholomorphic equivalence between $D_0$ and $E$.

In order to compute $\ti B_{D_0}(b^*, u^*)$ and $\ti g_{D_0}(b^*)$, by using the transformation rules of $\ti B$ and $\ti g$ under biholomorphism, it is equivalent of computing $\ti B_{E}\big(0,\Psi'(b^*)u^*\big)$ and $\ti g_{E}(0)\big|\det \Psi'(b^*)\big|^2$ respectively. One easily sees that $\Psi'(b^*)=\diag \{-1/2, 1\}$, and hence
\begin{align*}
&\lim_{\substack{z\to p \\ z\in \Ga}}\big|\pi_{1/d(z)}(u)\big|^{-1}\ti B_{\Om}(z,u)=\ti B_{E}\bigg(0,\big(-u^*_1/2, u_2^*\big)\bigg),\\
&\lim_{\substack{z\to p \\ z\in \Ga}} \big(d(z)\big)^{\sum_{j=1}^n 2/m_j} \ti g_{\Om}(z)=\dfrac{1}{4}\, \ti g_E(0).
\end{align*}

As we reduced our problem to a computation on the bounded Reinhardt domain $E$, first note that the monomials $z^{\al}$ form a complete orthogonal system in $A^2(E)$. Therefore the Bergman kernel on $E$ is given by the following power series
\begin{align}\label{eq 8}
K_E(z)=a_0+b_1 |z_1|^2+ b_2 |z_2|^2+ c_{20}|z_1|^4+ c_{11}|z_1|^2|z_2|^2+ c_{02}|z_2|^4+ \cdots,
\end{align}
where $a_0=1/\text{Vol}\,E$, $b_i=1/\|z_i\|_E^2$ and $c_{jk}=1/\|z_1^j z_2^k\|_E^2$. For our purpose, since we need to compute the partial derivatives of $K_E$ up to the fourth order and then evaluate them at $0$, we may truncate the above series after the fourth order terms. The remainder of the proof will follow from a bare-hand computation as outlined below:

Recall that
\begin{align}\label{eq 12}
\ti g^{ E}_{\al\ov \be}(z) &=3\,g^{ E}_{\al\ov \be}(z)-\Ric^{ E}_{\al\ov \be}(z)\\ \nonumber
&=3\, \dfrac{\pa^2}{\pa z_{\al} \pa \ov z_{\be}}\log K_E(z)+ \dfrac{\pa^2}{\pa z_{\al} \pa \ov z_{\be}}\log \det G_E(z).
\end{align}
We will denote
\begin{align}\label{eq 7}
A_E(z):=\det G_E(z)=\dfrac{\pa^2 \log K_E}{\pa z_{1} \pa \ov z_{1}}\dfrac{\pa^2 \log K_E}{\pa z_{2} \pa \ov z_{2}}-\dfrac{\pa^2 \log K_E}{\pa z_{1} \pa \ov z_{2}}\dfrac{\pa^2 \log K_E}{\pa z_{2} \pa \ov z_{1}}\bigg|_z,
\end{align}
and then by the chain rule
\begin{align}\label{eq 11}
\dfrac{\pa^2}{\pa z_{\al} \pa \ov z_{\be}}\log \det G_E=\dfrac{1}{A_E} \dfrac{\pa^2 A_E}{\pa z_{\al} \pa \ov z_{\be}}- \dfrac{1}{A^2_E}\dfrac{\pa A_E}{\pa z_{\al}}\dfrac{\pa A_E}{\pa \ov z_{\be}}.
\end{align}
Using the expression (\ref{eq 7}) and computing successive derivatives of $A_E$ by the product rule, we obtain
\begin{align}\label{eq 9}
\dfrac{\pa A_E}{\pa z_{\al}} &=\dfrac{\pa^2 \log K_E}{\pa z_{1} \pa \ov z_{1}}\dfrac{\pa^3 \log K_E}{\pa z_{\al} \pa z_2 \pa \ov z_2}+\dfrac{\pa^2 \log K_E}{\pa z_{2} \pa \ov z_{2}}\dfrac{\pa^3 \log K_E}{\pa z_{\al} \pa z_1 \pa \ov z_1}-\dfrac{\pa^2 \log K_E}{\pa z_{1} \pa \ov z_{2}}\dfrac{\pa^3 \log K_E}{\pa z_{\al} \pa z_2 \pa \ov z_1}\\ \nonumber
&-\dfrac{\pa^2 \log K_E}{\pa z_{2} \pa \ov z_{1}}\dfrac{\pa^3 \log K_E}{\pa z_{\al} \pa z_1 \pa \ov z_2},\\ \nonumber
\dfrac{\pa^2 A_E}{\pa z_{\al} \pa \ov z_{\be}} &=\dfrac{\pa^2 \log K_E}{\pa z_{1} \pa \ov z_{1}}\dfrac{\pa^4 \log K_E}{\pa z_{\al} \pa z_2 \pa \ov z_2 \pa \ov z_{\be}}+\dfrac{\pa^3 \log K_E}{\pa z_1 \pa\ov z_1 \pa \ov z_{\be}} \dfrac{\pa^3 \log K_E}{\pa z_{\al} \pa z_2 \pa \ov z_2}+ \dfrac{\pa^3 \log K_E}{\pa z_{\al} \pa z_1 \pa \ov z_1} \dfrac{\pa^3 \log K_E}{\pa z_2 \pa\ov z_2 \pa \ov z_{\be}}\\ \nonumber
&+\dfrac{\pa^2 \log K_E}{\pa z_{2} \pa \ov z_{2}}\dfrac{\pa^4 \log K_E}{\pa z_{\al} \pa z_1 \pa \ov z_1 \pa \ov z_{\be}}-\dfrac{\pa^2 \log K_E}{\pa z_{1} \pa \ov z_{2}}\dfrac{\pa^4 \log K_E}{\pa z_{\al} \pa z_2 \pa \ov z_1 \pa \ov z_{\be}}-\dfrac{\pa^3 \log K_E}{\pa z_1 \pa \ov z_2 \pa \ov z_{\be}} \dfrac{\pa^3 \log K_E}{\pa z_{\al} \pa z_2 \pa \ov z_1}\\ \nonumber
&-\dfrac{\pa^3 \log K_E}{\pa z_{\al} \pa z_1 \pa \ov z_2} \dfrac{\pa^3 \log K_E}{\pa z_2 \pa\ov z_1 \pa \ov z_{\be}}-\dfrac{\pa^2 \log K_E}{\pa z_{2} \pa \ov z_{1}}\dfrac{\pa^4 \log K_E}{\pa z_{\al} \pa z_1 \pa \ov z_2 \pa \ov z_{\be}}.
\end{align}
We then consider partial derivatives of $\log K_E$ using the expression of $K_E$ as in (\ref{eq 8}) and evaluate them at $0$ to get
\begin{align}\label{eq 10}
\hspace{0.1mm} &\dfrac{\pa^2 \log K_E}{\pa z_{i} \pa \ov z_{j}}(0)=\dfrac{1}{a_0}\big(b_1 \de_{ij1}+b_2 \de_{ij2}\big), \enskip\quad \dfrac{\pa^3 \log K_E}{\pa z_i \pa  z_j \pa \ov z_k}(0)=0,\\ \nonumber
&\dfrac{\pa^4 \log K_E}{\pa z_i \pa z_j \pa \ov z_k \pa \ov z_l}(0)=\dfrac{1}{a_0}\big[c_{20}\de_{ijkl1}+c_{11}(1-\de_{ij})\de_{\{i,j\},\{k,l\}}+c_{02}\de_{ijkl2}\big]\\ \nonumber
&\quad \quad-\dfrac{1}{a_0^2}\big[(b_1 \de_{jl1}+b_2 \de_{jl2})(b_1 \de_{ik1}+b_2 \de_{ik2})+(b_1 \de_{il1}+b_2 \de_{il2})(b_1 \de_{jk1}+b_2 \de_{jk2})\big]
\end{align}
for $i,j,k,l\in \{1,2\}$. Similar to the Kronecker delta, here $\de_{\cdots}$ is 1 if all the entries in its subscript are the same, otherwise it is 0. Next, we use these values in Eqs.~(\ref{eq 7}) and (\ref{eq 9}) to obtain
\begin{align*}
& A_E(0)=\dfrac{b_1 b_2}{a_0^2}, \quad \dfrac{\pa A_E}{\pa z_{\al}}(0)=0, \quad \dfrac{\pa^2 A_E}{\pa z_1 \pa \ov z_2}(0)=0,\\
& \dfrac{\pa^2 A_E}{\pa z_1 \pa \ov z_1}(0)=\dfrac{b_1 c_{11}+b_2 c_{20}}{a_0^2}- \dfrac{3\,b_1^2 b_2}{a_0^3}, \quad \dfrac{\pa^2 A_E}{\pa z_2 \pa \ov z_2}(0)=\dfrac{b_2 c_{11}+b_1 c_{02}}{a_0^2}- \dfrac{3\,b_1 b_2^2}{a_0^3}.
\end{align*}
Therefore, by (\ref{eq 11}),
\begin{align*}
\Ric^E_{1\ov 1}(0)=\dfrac{3\,b_1}{a_0}-\dfrac{b_1 c_{11}+b_2 c_{20}}{b_1 b_2}, \quad \Ric^E_{2\ov 2}(0)=\dfrac{3\,b_2}{a_0}-\dfrac{b_2 c_{11}+b_1 c_{02}}{b_1 b_2},\quad \Ric^E_{1\ov 2}(0)=0.
\end{align*}
Similarly, from the first identity in (\ref{eq 10}), one obtains
\begin{align*}
g^E_{1\ov 1}(0)=\dfrac{b_1}{a_0}, \quad g^E_{2\ov 2}(0)=\dfrac{b_2}{a_0},\quad g^E_{1\ov 2}(0)=0.
\end{align*}
Now the proof can be completed by a routine calculation implementing Eq.~(\ref{eq 12}).
\end{proof}

\noindent \textbf{Concluding remarks:} In \cite{Borah-Kar-2}, the authors studied the boundary behavior of the Gaussian curvature of the Kobayashi--Fuks metric on smoothly bounded planar domains after localising the Gaussian curvature near the boundary points. In the process of this localization, the Gaussian curvature was expressed in terms of some maximal domain functions related to the Kobayashi--Fuks metric. One could have then tried to study the boundary limits of the Gaussian curvature on bounded h-extendible domains by studying the boundary limits of those maximal domain functions using similar techniques employed in this article. But in dimension one, since the set of domains with smooth boundary coincides with the set of h-extendible domains, that would have yielded nothing new! In higher dimensions, I strongly feel that one can express the holomorphic sectional curvature and the Ricci curvature of the Kobayashi--Fuks metric in terms of similar maximal domain functions, which would help us not only in obtaining the localization results but also in deriving the boundary asymptotics of these associated curvatures on h-extendible domains.
\\

\noindent \textbf{Acknowledgements} The author would like to thank D. Borah for suggesting this problem and P. Mahajan for her constant encouragement, support and guidance.\\

\noindent \textbf{Funding} The author was supported by the National Board for Higher Mathematics (NBHM) Post-doctoral Fellowship (File no. 0204/10(7)/2023/R{\&}D-II/2777) and the Post-doctoral program at Indian Institute of Science, Bangalore.\\
 
\noindent \textbf{Data availability} Data sharing is not applicable to this article as no datasets were generated or analysed during the current study.\\

\noindent \textbf{Conflict of interest} The author declares that he has no conflict of interest.


\begin{bibdiv}
\begin{biblist}


\bib{BSY-1995}{article}{
   author={Boas, H. P.},
   author={Straube, E. J.},
   author={Yu, J.},
   title={Boundary limits of the Bergman kernel and metric},
   journal={Michigan Math. J.},
   volume={42},
   date={1995},
   number={3},
   pages={449--461},
   doi={10.1307/mmj/1029005306},
}


\bib{Borah-Kar-2}{article}{
   author={Borah, D.},
   author={Kar, D.},
   title={Some remarks on the Kobayashi-Fuks metric on strongly pseudoconvex
   domains},
   journal={J. Math. Anal. Appl.},
   volume={512},
   date={2022},
   number={2},
   pages={Paper No. 126162, 24},
   doi={10.1016/j.jmaa.2022.126162},
}


\bib{BV-ns}{article}{
author={Borah, D.},
author={Verma, K.},
title={Narasimhan--Simha type metrics on strongly pseudoconvex domains in $\mathbb{C}^n$},
journal={arXiv:2105.07723},
}


\bib{Catlin2}{article}{
   author={Catlin, D.},
   title={Boundary invariants of pseudoconvex domains},
   journal={Ann. of Math. (2)},
   volume={120},
   date={1984},
   number={3},
   pages={529--586},
   doi={10.2307/1971087},
}


\bib{Catlin-1989}{article}{
   author={Catlin, D.},
   title={Estimates of invariant metrics on pseudoconvex domains of
   dimension two},
   journal={Math. Z.},
   volume={200},
   date={1989},
   number={3},
   pages={429--466},
   doi={10.1007/BF01215657},
}


\bib{Conrad}{article}{
   author={Conrad, M.},
   title={Nicht isotrope Abschätzungen für lineal konvexe Gebiete endlichen Typs},
   journal={Dissertation, Universität Wuppertal, 2002},
}


\bib{D'Angelo-1982}{article}{
   author={D'Angelo, J.},
   title={Real hypersurfaces, orders of contact, and applications},
   journal={Ann. of Math. (2)},
   volume={115},
   date={1982},
   number={3},
   pages={615--637},
   doi={10.2307/2007015},
}



\bib{Diederich-1970}{article}{
   author={Diederich, K.},
   title={Das Randverhalten der Bergmanschen Kernfunktion und Metrik in
   streng pseudo-konvexen Gebieten},
   language={German},
   journal={Math. Ann.},
   volume={187},
   date={1970},
   pages={9--36},
   doi={10.1007/BF01368157},
}

\bib{Died-Forn}{article}{
   author={Diederich, K.},
   author={Forn\ae ss, J. E.},
   title={Boundary behavior of the Bergman metric},
   journal={Asian J. Math.},
   volume={22},
   date={2018},
   number={2},
   pages={291--298},
   doi={10.4310/ajm.2018.v22.n2.a6},
}


\bib{Died-Herbort-1994}{article}{
   author={Diederich, K.},
   author={Herbort, G.},
   title={Pseudoconvex domains of semiregular type},
   conference={
      title={Contributions to complex analysis and analytic geometry},
   },
   book={
      series={Aspects Math., E26},
      publisher={Friedr. Vieweg, Braunschweig},
   },
   date={1994},
   pages={127--161},
}


\bib{Dinew11}{article}{
   author={Dinew, \.{Z}.},
   title={On the Bergman representative coordinates},
   journal={Sci. China Math.},
   volume={54},
   date={2011},
   number={7},
   pages={1357--1374},
   doi={10.1007/s11425-011-4243-4},
}

\bib{Dinew13}{article}{
   author={Dinew, \.{Z}.},
   title={On the completeness of a metric related to the Bergman metric},
   journal={Monatsh. Math.},
   volume={172},
   date={2013},
   number={3-4},
   pages={277--291},
   doi={10.1007/s00605-013-0501-6},
}


\bib{Fefferman-1974}{article}{
   author={Fefferman, C.},
   title={The Bergman kernel and biholomorphic mappings of pseudoconvex
   domains},
   journal={Invent. Math.},
   volume={26},
   date={1974},
   pages={1--65},
   doi={10.1007/BF01406845},
}


\bib{Fuks66}{article}{
   author={Fuks, B. A.},
   title={The Ricci curvature of the Bergman metric invariant with respect
   to biholomorphic mappings},
   language={Russian},
   journal={Dokl. Akad. Nauk SSSR},
   volume={167},
   date={1966},
   pages={996--999},
}


\bib{Gr}{article}{
   author={Graham, I.},
   title={Boundary behavior of the Carath\'{e}odory and Kobayashi metrics on
   strongly pseudoconvex domains in $C^{n}$ with smooth boundary},
   journal={Trans. Amer. Math. Soc.},
   volume={207},
   date={1975},
   pages={219--240},
   doi={10.2307/1997175},
}


\bib{Herbort-German-1983}{article}{
   author={Herbort, G.},
   title={\"{U}ber das Randverhalten der Bergmanschen Kernfunktion und Metrik
   f\"{u}r eine spezielle Klasse schwach pseudokonvexer Gebiete des ${\bf
   C}^{n}$},
   language={German},
   journal={Math. Z.},
   volume={184},
   date={1983},
   number={2},
   pages={193--202},
   doi={10.1007/BF01252857},
}


\bib{Jarn-Nik-2002}{article}{
   author={Jarnicki, M.},
   author={Nikolov, N.},
   title={Behavior of the Carath\'{e}odory metric near strictly convex boundary
   points},
   journal={Univ. Iagel. Acta Math.},
   number={40},
   date={2002},
   pages={7--12},
}


\bib{Jarn-Pflug-2013}{book}{
   author={Jarnicki, M.},
   author={Pflug, P.},
   title={Invariant distances and metrics in complex analysis},
   series={De Gruyter Expositions in Mathematics},
   volume={9},
   edition={Second extended edition},
   publisher={Walter de Gruyter GmbH \& Co. KG, Berlin},
   date={2013},
   pages={xviii+861},
   doi={10.1515/9783110253863},
}


\bib{Kar}{article}{
author={Kar, D.},
title={Existence of Geodesic Spirals for the Kobayashi--Fuks Metric on
   Planar Domains},
   journal={Complex Anal. Oper. Theory},
   volume={17},
   date={2023},
   number={4},
   pages={Paper No. 46},
   doi={10.1007/s11785-023-01355-7},
}
   


\bib{Kim-Park}{article}{
   author={Kim, J. J.},
   author={Park, S. H.},
   title={The boundary behavior between the Kobayashi-Royden and
   Carath\'{e}odory metrics on strongly pseudoconvex domain in $\bold C^n$},
   journal={Honam Math. J.},
   volume={19},
   date={1997},
   number={1},
   pages={81--86},
}


\bib{Klembeck-1978}{article}{
   author={Klembeck, P. F.},
   title={K\"{a}hler metrics of negative curvature, the Bergmann metric near the
   boundary, and the Kobayashi metric on smooth bounded strictly
   pseudoconvex sets},
   journal={Indiana Univ. Math. J.},
   volume={27},
   date={1978},
   number={2},
   pages={275--282},
   doi={10.1512/iumj.1978.27.27020},
}


\bib{Kob59}{article}{
   author={Kobayashi, S.},
   title={Geometry of bounded domains},
   journal={Trans. Amer. Math. Soc.},
   volume={92},
   date={1959},
   pages={267--290},
   doi={10.2307/1993156},
}

\bib{Krantz-Yu}{article}{
   author={Krantz, S. G.},
   author={Yu, J.},
   title={On the Bergman invariant and curvatures of the Bergman metric},
   journal={Illinois J. Math.},
   volume={40},
   date={1996},
   number={2},
   pages={226--244},
}

\bib{Lee-2001}{article}{
   author={Lee, S.},
   title={Asymptotic behavior of the Kobayashi metric on certain
   infinite-type pseudoconvex domains in ${\bf C}^2$},
   journal={J. Math. Anal. Appl.},
   volume={256},
   date={2001},
   number={1},
   pages={190--215},
   doi={10.1006/jmaa.2000.7307},
}

\bib{McNeal-1992}{article}{
   author={McNeal, J. D.},
   title={Lower bounds on the Bergman metric near a point of finite type},
   journal={Ann. of Math. (2)},
   volume={136},
   date={1992},
   number={2},
   pages={339--360},
   doi={10.2307/2946608},
}


\bib{Nikolov-2000}{article}{
   author={Nikolov, N.},
   title={Continuity and boundary behavior of the Carath\'{e}odory metric},
   language={Russian, with Russian summary},
   journal={Mat. Zametki},
   volume={67},
   date={2000},
   number={2},
   pages={230--240},
   issn={0025-567X},
   translation={
      journal={Math. Notes},
      volume={67},
      date={2000},
      number={1-2},
      pages={183--191},
   },
   doi={10.1007/BF02686245},
}


\bib{Ramadanov-1967}{article}{
   author={Ramadanov, I.},
   title={Sur une propri\'{e}t\'{e} de la fonction de Bergman},
   language={French},
   journal={C. R. Acad. Bulgare Sci.},
   volume={20},
   date={1967},
   pages={759--762},
}


\bib{Warsz}{article}{
   author={Warszawski, T.},
   title={Boundary behavior of the Kobayashi distance in pseudoconvex
   Reinhardt domains},
   journal={Michigan Math. J.},
   volume={61},
   date={2012},
   number={3},
   pages={575--592},
   doi={10.1307/mmj/1347040260},
}

\bib{Yu-1994}{article}{
   author={Yu, J.},
   title={Peak functions on weakly pseudoconvex domains},
   journal={Indiana Univ. Math. J.},
   volume={43},
   date={1994},
   number={4},
   pages={1271--1295},
   doi={10.1512/iumj.1994.43.43055},
}


\bib{Yu-1995}{article}{
   author={Yu, J.},
   title={Weighted boundary limits of the generalized Kobayashi-Royden
   metrics on weakly pseudoconvex domains},
   journal={Trans. Amer. Math. Soc.},
   volume={347},
   date={1995},
   number={2},
   pages={587--614},
   doi={10.2307/2154903},
}


\end{biblist}
\end{bibdiv}
\end{document}